\documentclass[a4paper,11pt]{amsart}
\usepackage[utf8]{inputenc}
\usepackage[a4paper]{geometry} 
\usepackage[english]{babel}

\usepackage{mathtools,amssymb,bm,amsthm,amsmath} 
\usepackage{amsfonts}
\usepackage{xcolor} 
\usepackage{graphicx} 
\usepackage[pdfencoding=auto]{hyperref} 
\usepackage{stmaryrd, enumitem}
\usepackage{dsfont}
\usepackage{comment}

\newtheorem{theorem}{Theorem}[section]
\newtheorem{proposition}[theorem]{Proposition}
\newtheorem{corollary}[theorem]{Corollary}
\newtheorem{lemma}[theorem]{Lemma}
\newtheorem{fact}[theorem]{Fact}

\newtheorem*{step1}{Step 1}
\newtheorem*{step2}{Step 2}
\newtheorem*{step3}{Step 3}
\newtheorem*{step4}{Step 4}
\newtheorem*{step5}{Step 5}
\newtheorem*{step6}{Step 6}

\theoremstyle{remark}

\theoremstyle{definition}
\newtheorem{definition}[theorem]{Definition}
\newtheorem{Remark}[theorem]{Remark}

\numberwithin{equation}{section}

\newcommand{\R}{\mathbb{R}}
\newcommand{\N}{\mathbb{N}}

\newcommand{\ind}{\mathds{1}}
\newcommand{\tild}{\widetilde}
\newcommand{\minus}{\setminus}
\newcommand{\vertiii}[1]{{\left\vert\kern-0.25ex\left\vert\kern-0.25ex\left\vert #1 
		\right\vert\kern-0.25ex\right\vert\kern-0.25ex\right\vert}}

\DeclareMathOperator{\Span}{span}
\DeclareMathOperator{\Iso}{Iso}
\DeclareMathOperator{\Emb}{Emb}
\DeclareMathOperator{\BM}{BM}
\DeclareMathOperator{\Age}{Age}
\DeclareMathOperator{\Ball}{B}
\DeclareMathOperator{\Sphere}{S}
\DeclareMathOperator{\supp}{supp}

\newcommand{\FLMT}{Ferenczi--L\'opez-Abad--Mbombo--Todorcevic}
\newcommand{\CDDK}{C\'uth--Dole\v{z}al--Doucha--Kurka}
\newcommand{\CdRD}{C\'uth--de Rancourt--Doucha}

\begin{document}
	\emergencystretch 3em

\title{Isometric rigidity and Fra\"iss\'e properties of Orlicz sequence spaces}

\author[N.~de~Rancourt]{No\'e de Rancourt}
\author[M.~Fakhoury]{Micheline Fakhoury}

\address[N.~de~Rancourt]{Universit\'e de Lille, CNRS, UMR 8524 -- Laboratoire Paul Painlev\'e, F-59000 Lille, France}
\email{nderancour@univ-lille.fr}

\address[M.~Fakhoury]{Universit\'e de Lille, CNRS, UMR 8524 -- Laboratoire Paul Painlev\'e, F-59000 Lille, France}
\email{micheline.fakhoury@univ-lille.fr}

\subjclass[2020]{Primary: 46B04; Secondary: 46B45, 46E30, 03C66, 03C35.}

\keywords{Fra\"iss\'e Banach spaces, Orlicz spaces, isometric rigidity, $\omega$-categoricity}

\thanks{Both authors acknowledge support from the Labex CEMPI (ANR-11-LABX-0007-01) and the CDP C2EMPI, together with the French State under the France-2030 programme, the University of Lille, the Initiative of Excellence of the University of Lille, the European Metropolis of Lille for their funding and support of the R-CDP-24-004-C2EMPI project.}

\date{}

\begin{abstract}
We provide an approximate version of a rigidity result by Randrianantoanina: for a large class of Orlicz sequence spaces, almost isometric embeddings almost preserve disjointness. In specific cases, we can even prove that such embeddings almost preserve basic vectors. As a consequence, we prove that some Orlicz sequences spaces are guarded Fra\"iss\'e but not $\omega$-categorical; moreover, they do not contain copies of $\ell_2$ and their age is not closed. This answers a question of C\'uth--de Rancourt--Doucha.
\end{abstract}

\maketitle

\section{Introduction}

\subsection{Fra\"iss\'e theory for Banach spaces}

In this paper, all Banach spaces will be separable. Say that a Banach space $X$ is \textit{ultrahomogeneous} if for every finite-dimensional subspace $E \subseteq X$, group $\Iso(X)$ of all (surjective linear) isometries of $X$ acts transitively on the set $\Emb(E, X)$ of all (linear) isometric embeddings $E \to X$. Hilbert spaces are obviously ultrahomogeneous, and an important question, raised by \FLMT{} \cite{FLMT}, is to know whether they are the only ones. This question is related to the well-known \textit{rotation problem} of Mazur, asking whether $\Iso(X)$ acts transitively on the unit sphere $S_X$ of $X$.

\smallskip

Several approximate versions of ultrahomogeneity have been considered, the most important one being this of \textit{Fra\"iss\'e Banach spaces}, introduced by \FLMT{} \cite{FLMT}. For $C \geqslant 1$, let $\Emb_C(E, X)$ denote the set of all linear injections $T \colon E \to X$ with $\max(\|T\|, \|T^{-1}\|) \leqslant C$. Say that a group $G$ acts $\varepsilon$-transitively on a metric space $M$ if for every $x, y \in M$, we have $d(y, Gx) \leqslant \varepsilon$.

\begin{definition}[\FLMT{} 2020]
A Banach space $X$ is \textit{weak Fra\"iss\'e} if for every finite-dimensional subspace $E\subseteq X$ and $\varepsilon > 0$, there exists $\delta > 0$ such that $\Iso(X)$ acts $\varepsilon$-transitively on $\Emb_{1+\delta}(E, X)$. It is \textit{Fra\"iss\'e} if it is weak Fra\"iss\'e with $\delta$ depending only on $\varepsilon$ and $\dim(E)$.
\end{definition}

\noindent Examples of infinite-dimensional Fra\"iss\'e Banach spaces include the Gurari\u{i} space \cite{Gurarij} and, as shown in \cite{FLMT}, the spaces $L_p([0, 1])$ for $1 \leqslant p < \infty$ with $p \notin 4, 6, 8, \ldots$ (the spaces $L_{2n}[0, 1]$ for integers $n\geqslant 2$ are known not to be Fra\"iss\'e, as a consequence of a result of Randrianantoanina \cite{BeataFraisse}). It is an important open question whether those examples are the only ones. It is not knwon whether the classes of Fra\"iss\'e and weak Fra\"iss\'e Banach spaces coincide.

\smallskip

The terminology \textit{Fra\"iss\'e} comes from \textit{Fra\"iss\'e theory}, a field that originated with the proof of the \textit{Fra\"iss\'e correspondence} \cite{Fraisse}, a model-theoretic result establishing a bijective correspondence between countable, first-order ultrahomogeneous structures, and certain classes of finitely generated structures called \textit{Fra\"iss\'e classes}. (We refer to \cite[Section 7.1]{Hodges} for more details on the basics of the theory.)  The definition of Fraïssé Banach spaces has been tailored so that a similar correspondance holds in this setting (see \cite[Corollary 2.27]{FLMT}). Beyond the Mazur rotation problem, the links between Fra\"iss\'e theory and topological dynamics of automorphism groups (see \cite{KPT}) also motivated the introduction of this notion, and were exploited in \cite{FLMT} to provide a new proof of extreme amenability of the groups $\Iso(L_p)$, $1 \leqslant p \neq 2 < \infty$ (a result by Giordano--Pestov \cite{GiordanoPestov}).

\smallskip


In a recent paper \cite{CDDK}, \CDDK{} introduced several Polish spaces coding the class of all separable Banach spaces, and proved that the Gurari\u{i} space and the spaces $L_p([0, 1])$, $1\leqslant  p < \infty$, all had $G_\delta$ isometry classes in those codings. They asked about the existence of other examples. A natural first step was to characterize Banach spaces having a $G_\delta$ isometry class; this was done by \CdRD{} \cite{CdRD}. The characterization they provided is a weakening of the weak Fra\"iss\'e property, and Banach spaces satisfying it were called \textit{guarded Fra\"iss\'e}.

\begin{definition}[C\'uth--de Rancourt--Doucha 2024+]
A separable Banach space $X$ is \textit{guarded Fra\"iss\'e} if for every finite-dimensional subspace $E\subseteq X$ and every $\varepsilon>0$, there exist $\delta>0$, and a finite dimensional subspace $F\subseteq X$ with $E \subseteq F$, such that $\Iso(X)$ acts $\varepsilon$-transitively on $\Emb_{1+\delta}(F,X)\restriction E$.
\end{definition}

\noindent Here, $\Emb_{1+\delta}(F,X)\restriction E$ denotes the set $\{T_{\restriction E} \mid T\in \Emb_{1+\delta}(F,X)\}$. Note that guarded Fra\"iss\'eness is a Banach space version of the notion of \textit{prehomogeneity} from model theory; see \cite{KrawczykKubis} for more information on this notion. There is also a version of the Fra\"iss\'e correspondence for guarded Fra\"iss\'e Banach spaces; see \cite[Theorem 3.9]{CdRD}.

\smallskip

As a consequence of \cite[Proposition 3.7]{FLMT}, all spaces $L_p([0, 1])$, $1 \leqslant p < \infty$, are guarded Fra\"iss\'e. The proof of the above cited result importantly relies on Banach--Lamperti's isometric rigidity theorem for $L_p$-spaces \cite{Lamperti}. Those spaces are actually even \textit{cofinally Fra\"iss\'e}, an intermediate notion between guarded Fra\"iss\'e and weak Fra\"iss\'e; see Definition \ref{def:CofFraisse} below. New examples of non-weak Fra\"iss\'e, cofinally Fra\"iss\'e Banach spaces were provided in \cite{CdRD}, namely the Lebesgue--Bochner spaces $L_p([0, 1], L_q([0, 1]))$, whenever $1 \leqslant p \neq q < \infty$ and the pair $(p, q)$ doesn't satisfy $q = 2$ or $p < q < 2$. The proof, once again, relies on an isometric rigidity result for those spaces due to Guerre--Raynaud \cite{GuerreRaynaud}. In general, good candidates for cofinal Fra\"iss\'eness are spaces having the same degree of isometric rigidity locally and globally.

\medskip

\subsection{Oligomorphy and related notions} One reason of the interest of model theorists in the Fra\"iss\'e construction is its usefulness for constructing $\omega$-categorical structures with prescribed properties. A countable, first-order structure is said to be \textit{$\omega$-categorical} if it is entirely determined, up to isomorphism, by its first-order theory. (We refer the reader to \cite[Section 7.3]{Hodges} for more details about this notion.) In continuous logic, a version of this notion exists for separable first order metric structures, and in particular for separable Banach spaces, see \cite{BYBHU}.  Stating the original definition would require much formalism; however, the following equivalent characterization of $\omega$-categorical Banach spaces (see \cite[Theorem 12.10 and Corollary 12.11]{BYBHU}) can be taken as a definition.


\begin{definition}
    A separable Banach space $X$ is \textit{$\omega$-categorical}, or \textit{oligomoprhic}, if for every $n \in \N$, the metric quotient $(\Ball_X)^n\sslash\Iso(X)$ is compact.
\end{definition}

\noindent Here, $\Ball_X$ denotes the closed unit ball of $X$ and $(\Ball_X)^n\sslash\Iso(X)$ denotes the metric quotient of $(\Ball_X)^n$ by the action of $\Iso(X)$ (see \cite[Definition 9.1]{CdRD} for a precise definition of this quotient). The alternative terminology \textit{oligomorphic} is motivated by the above characterization; following \cite{FerencziRincon}, we will favour it over \textit{$\omega$-cateogorical} since it also makes sense for nonseparable Banach spaces. Examples of oligomorphic Banach spaces include $\mathcal{C}(2^\N)$, the Gurari\u{i} space \cite{BenYaacovHenson}, all spaces $L_p([0, 1])$ \cite{BYBHU}, and all spaces $L_p([0, 1], L_q([0, 1]))$ \cite{HensonRaynaud}.

\smallskip

Oligomorphic Banach spaces have the following interesting property: if a Banach space $Y$ is finitely representable in an oligomorphic Banach space $X$, then $Y$ can be isometrically embedded in $X$ (see \cite[Theorem 9.9]{CdRD}). This has the two following consequences.
\begin{itemize}
    \item Every oligomorphic Banach space contains an isometric copy of $\ell_2$ (this is a consequence of the Dvoretky theorem).
    \item Oligomorphic Banach spaces have closed age. The \textit{age} of a Banach space $X$, denoted by $\Age(X)$, is the class of all of its finite-dimensional subspaces, considered up to isometry. $\Age(X)$ is seen as a subset of the class of all finite-dimensional Banach spaces, endowed with the Banach-Mazur distance. So saying that $\Age(X)$ is closed means that whenever a finite-dimensional space $E$ admits almost isometric copies in $X$, then it actually has an isometric copy in $X$.
\end{itemize}

\smallskip

It is a standard fact that all countable ultrahomogeneous first-order structures in a finite relational language are $\omega$-categorical. Banach spaces behave quite similarly as structures on a finite relational language, and as expected, all Fra\"iss\'e Banach spaces are oligomorphic (see \cite[Theorem 5.13]{FerencziRincon}). In particular, they contain isometric copies of $\ell_2$ and have closed age; actually, a Banach space is Fra\"iss\'e iff it is weak Fra\"iss\'e and has closed age, as shown in \cite[Theorem 2.12]{FLMT}. Similarly, in \cite{CdRD}, \CdRD{} asked whether guarded Fra\"iss\'e Banach spaces should necessarily be oligomorphic, have closed age, and contain isometric copies of $\ell_2$. The main goal of the present paper is to provide a negative answer to those questions.

\smallskip

Note that there are simple examples of discrete prehomogeneous structures that are not $\omega$-categorical. This is for instance the case of the infinite linear graph (with vertex set $\mathbb{Z}$, and where two integers are connected iff they are consecutive), see \cite[Example 5.4]{KrawczykKubis}. A natural idea to find counterexamples to the above questions is to look at Banach spaces presenting similar rigidity properties as this graph. We will find such spaces among \textit{Orlicz sequences spaces}, a class of spaces generalizing $\ell_p$ spaces.

\medskip

\subsection{Orlicz spaces} We start with recalling the definition of Orlicz spaces. In what follows, we fix a $\sigma$-finite measure space $(\Omega, \mu)$ and we denote by $L_0(\mu)$ the vector space of all measurable functions $f \colon \Omega \to \R$, where almost everywhere equal functions are identified. We also fix an \textit{Orlicz function}, that is, a convex function $M \colon [0, +\infty) \to [0, +\infty)$ with $M(0) = 0$ and $\lim_{t \to \infty} M(t) = \infty$.

\begin{definition}
    For $f \in L_0(\mu)$, define the \textit{Luxemburg norm} of $f$ by $$\|f\|_M \coloneq \inf\left\{\rho > 0 \,\left|\, \int M\left(\frac{|f|}{\rho}\right)d\mu \leqslant 1\right.\right\},$$
    where by convention, $\inf \varnothing = + \infty$. The \textit{Orlicz space} associated to $M$ on $(\Omega, \mu)$ is defined as $$L_M(\mu) \coloneq \{f \in L_0(\mu) \mid \|f\|_M < \infty\}.$$
    We equip it with the norm $\|\cdot\|_M$.
\end{definition}

\noindent It is a classical result that Orlicz spaces are Banach spaces. When $M(t) = t^p$ for $1 \leqslant p < \infty$, the space $L_M(\mu)$ is simply $L_p(\mu)$. In the special case when $(\Omega, \mu)$ is $\N$ (resp. an $n$-elements set) with the counting measure, the space $L_M(\mu)$ is denoted by $\ell_M$ (resp. $\ell_M^n$). Spaces $\ell_M$ are called \textit{Orlicz sequence spaces}. The space $\ell_M$ is separable iff the function $M$ satisfies the \textit{$\Delta_2$-condition at $0$}, that is, $$\limsup_{t \to 0}\frac{M(2t)}{M(t)} < \infty.$$
Whenever this condition holds, the canonical basis $(e_n)$ of $c_{00}$ is a $1$-symmetric basis of~$\ell_M$. We refer the reader to \cite{RaoRen} for the general theory of Orlicz spaces, and to \cite[Chapter 4]{LT} for the special case of Orlicz sequences spaces.

\smallskip

For $f\in L_0(\mu)$, the \textit{support} of $f$ is defined as $\supp(f) \coloneq \{x \in \Omega \mid f(x) \neq 0\}$. This is a subset of $\Omega$ defined up to zero measure. Two vectors $f, g \in L_0(\mu)$ are said to be \textit{disjoint} whenever $\mu(\supp(f) \cap \supp(g)) = 0$. The classical Banach--Lamperti theorem \cite{Lamperti} asserts that every isometric embedding $T \colon L_p(\mu) \to L_p(\mu)$, where $1 \leqslant p \neq 2 < \infty$, \textit{preserves disjointness}, i.e. $T(f)$ and $T(g)$ are disjoint whenever $f$ and $g$ are. This result has been generalized by Randrianantoanina \cite{Beata} to a large class of Orlicz spaces on $[0, 1]$.

\begin{theorem}[Randrianantoanina 1998]\label{thm:Beata}
    Assume that $M$ is $\mathcal{C}^2$ on $(0, +\infty)$, satisfies the $\Delta_{2+}$-condition at infinity, that $M'(0) = 0$, $M''(t) > 0$ whenever $t > 0$, and $M''(t)$ tends to either $0$ or $\infty$ when $t \to 0$. Then all isometric embeddings $L_M([0, 1]) \to L_M([0, 1])$ preserve disjointness.
\end{theorem}

\noindent The \textit{$\Delta_{2+}$-condition at infinity} in use here is a strengthening of the classical \textit{$\Delta_2$-condition at infinity} (which ensures separability of $L_M([0, 1])$). We won't explicit this condition, which ensures both a polynomial-like growth rate and regularity at infinity.

\medskip

\subsection{Results} Our main result is that some Orlicz sequence spaces are cofinally Fra\"iss\'e. Set aside the fact that Theorem \ref{thm:Beata} only holds for function spaces over $[0, 1]$, another issue is that proving cofinal Fra\"iss\'eness usually requires an approximate rigidity result for \textit{almost isometric} embeddings (i.e. elements of $\Emb_{1+\delta}(E, X)$ for small $\delta$). In the case of spaces $L_p([0, 1])$ or $L_p([0, 1], L_q([0, 1]))$, rigidity results for \textit{exact} isometric embeddings are enough, because oligomorphy allows one to pass from exact isometric to almost isometric embeddings for free, see \cite[Lemma 9.17]{CdRD}. We cannot hope for such a shortcut in our case, given that our goal is precisely to find non-oligomorphic examples.

\smallskip

Our first result is a result of almost disjointness preservation for almost isometric embeddings in a large class of Orlicz spaces.

\begin{theorem}\label{thm:MainDisjointnessPreservation}
    Let $M_1$ and $M_2$ be Orlicz functions. Suppose that for $i\in\{1, 2\}$, $M_i$ is $\mathcal{C}^3$ on $[0, +\infty)$,  $M_i^{(3)}$ is increasing, $M_i(t) = o(t^3)$ when $t \to 0$, and $M_i$ satisfies the $\Delta_{2++}$-condition at $0$. Let $X = \ell_{M_1}$ or $\ell_{M_1}^n$ for some $n \in \N$. Then for every $\varepsilon > 0$, there exists $\delta > 0$ such that for every $T \in \Emb_{1+\delta}(X, \ell_{M_2})$ and every disjoint $f, g \in X$, there exist disjoint $\widetilde{f}, \widetilde{g} \in \ell_{M_2}$ such that $\|T(f) - \widetilde{f}\| \leqslant \varepsilon$ and $\|T(g) - \widetilde{g}\| \leqslant \varepsilon$.
\end{theorem}

\noindent The \textit{$\Delta_{2++}$-condition at $0$} is a strengthening of the $\Delta_2$-condition at $0$ introduced above (in particular, the spaces we consider are separable). It ensures both a polynomial-like growth and regularity in a neighborhood of $0$; see Definition \ref{def:Delta2PP}. It resembles Randrianantoanina's $\Delta_{2+}$-condition at infinity from Theorem \ref{thm:Beata}, but is somewhat stricter in terms of regularity. Our proof itself follows the same ideas as Randrianantoanina's proof of Theorem \ref{thm:Beata}. As a consequence of the Banach--Lamperti theorem and a result of Schechtman \cite{Schechtman}, the conclusion of Theorem \ref{thm:MainDisjointnessPreservation}
is known to hold when $M_1(t) = M_2(t) = t^p$ for \textit{any} $1 \leqslant p \neq 2 < \infty$, and not just for $3 < p < \infty$ as our assumptions on the $M_i$'s seem to suggest. These assumptions are likely far from being optimal, and may be due to a defect in our methods. We hope to improve them in a future version of this preprint. Theorem \ref{thm:MainDisjointnessPreservation} admits the following immediate corollary.

\begin{corollary}\label{cor:DisjointnessPreservation}
    Let $M_1$ and $M_2$ be Orlicz functions satisfying the same assumptions as in Theorem \ref{thm:MainDisjointnessPreservation}, and let $X = \ell_{M_1}$ or $\ell_{M_1}^n$ for some $n \in \N$. Then all isometric embeddings $X \to \ell_{M_2}$ preserve disjointness.
\end{corollary}

For arbitrary Orlicz functions, we cannot hope to get better than disjointness preservation results in general, since the $\ell_p$ spaces have a wealth of self-isometric embeddings (the canonical basis can be mapped to any pairwise disjoint family of norm-$1$ vectors). However, if we strengthen our conditions to avoid $\ell_p$ spaces, we can get an even stronger rigidity result: almost isometric embeddings almost preserve basic vectors, up to a change of signs. Here, by a \textit{basic vector} of $\ell_M^n$ or $\ell_M$, we mean a vector of the canonical basis.

\begin{theorem}\label{thm:MainBasisPreservation}
     Let $M$ be an Orlicz function satisfying the same assumptions as in Theorem \ref{thm:MainDisjointnessPreservation}. Suppose additionally that $M(1) = 1$, and for every $\varepsilon > 0$, there exists $\alpha(\varepsilon) > 1$ such that for every $0< t,\, u \leqslant 1 - \varepsilon$, it holds that $M(tu) \geqslant \alpha(\varepsilon)M(t)M(u)$. Let $X = \ell_{M}$ or $\ell_{M}^n$ for some $n \in \N$. Then for every $\varepsilon > 0$, there exists $\delta > 0$ such that for every $T \in \Emb_{1 + \delta}(X, \ell_M)$ and every basic vector $e$ of $X$, there exist a basic vector $f$ of $\ell_M$ and a sign $\theta \in \{-1, 1\}$ for which $\|T(e) - \theta f\| \leqslant \varepsilon$.
\end{theorem}

\noindent Note that the assumptions of Theorem \ref{thm:MainBasisPreservation} are for instance satisfied when $M(t) = t^pe^{t-1}$ for any $3 < p < \infty$. Theorem \ref{thm:MainBasisPreservation} admits the two following corollaries.

\begin{corollary}\label{cor:BasisPreservation}
    Let $M$ be an Orlicz function satisfying the same assumptions as in Theorem \ref{thm:MainBasisPreservation}. Let $X = \ell_{M}$ or $\ell_{M}^n$ for some $n \in \N$. Then all isometric embeddings $X \to \ell_M$ map basic vectors to multiples of basic vectors.
\end{corollary}

\begin{corollary}\label{cor:OrliczFraisse}
    Let $M$ be an Orlicz function satisfying the same assumptions as in Theorem \ref{thm:MainBasisPreservation}. Then the space $\ell_M$ is cofinally Fra\"iss\'e.
\end{corollary}

The following result finally answers \cite[Questions 5 and 6]{CdRD}.

\begin{theorem}\label{thm:Counterexample}
    Let $M(t) = t^pe^{t-1}$, where $3 < p < \infty$. Then the Banach space $\ell_M$ is guarded Fra\"iss\'e, but it is not oligomorphic, does not contain isometric (and even isomorphic) copies of $\ell_2$, and has a non-closed age.
\end{theorem}

The rest of this paper is organized as follow. In Section \ref{sec:2}, we prove Theorem \ref{thm:MainDisjointnessPreservation}. In Section \ref{sec:3}, we prove Theorem \ref{thm:MainBasisPreservation}. Finally, in Section \ref{sec:4}, we review some notions related to Fra\"iss\'e theory and oligomorphy, and we prove Corollary \ref{cor:OrliczFraisse} and \ref{thm:Counterexample}.

\bigskip

\section{Disjointness preservation results}\label{sec:2}

We start with defining the $\Delta_{2++}$-condition used in the statement of Theorem \ref{thm:MainDisjointnessPreservation}.

\begin{definition}\label{def:Delta2PP}
    An Orlicz function $M$ of class $\mathcal{C}^3$ on $(0, +\infty)$ is said to satisfy the \textit{$\Delta_{2++}$-condition at $0$} if for every $i \in \{1, 2, 3\}$, it holds that $$\limsup_{t \to 0}\frac{tM^{(i)}(t)}{M^{(i-1)}(t)} < \infty.$$
\end{definition}

It is easily seen that the above condition for $i=1$ implies the $\Delta_2$-condition at $0$. In particular, it ensures mild growth in a neighborhood of $0$. However, the same condition of $i = 2$ and $3$ also ensures some degree of regularity at the neighborhood for $0$, avoiding too fast oscillations. For a more detailed analysis, we refer the reader to \cite[Section 2]{Beata}, where a related condition, the $\Delta_{2+}$-condition at infinity, is discussed.

\smallskip

This section is devoted to the proof of Theorem \ref{thm:MainDisjointnessPreservation}. For simplicity, an Orlicz function satisfying the assumptions of this theorem will be called \textit{good}. Note that if $M$ is a good Orlicz function, then for every $0 \leqslant i \leqslant 3$, $M^{(i)}(0) = 0$ and $M^{(i)}(t) > 0$ whenever $t > 0$; in particular, the function $t \mapsto M(|t|)$ is $\mathcal{C}^3$ on $\R$. Moreover, there exists $K>1$ such that for any $0<t\leqslant 15 $ and $1 \leqslant i \leqslant 3$, we have 
$$\frac{tM^{(i)}(t)}{M^{(i-1)}(t)}\leqslant K.$$

For the five first steps of this proof, we fix a good Orlicz function $M$. Let $f,g\in \ell_M$ be such that $4/5\leqslant \Vert f\Vert\leqslant5/4 $ and $4/5\leqslant\Vert g\Vert\leqslant5/4 $. 

\smallskip
We define $F_{f,g} : (-1/2,1/2)\times (1/8,2) \to \mathbb{R}$  by \[F_{f,g}(\alpha,\eta) =\sum_{k=1}^\infty a_k(\alpha,\eta)-1,\]where $a_k(\alpha,\eta)=M\biggl(\frac{\vert f(k)+\alpha g(k)\vert }{\eta}\biggr)$. Note that for any $k \in \mathbb{N}$, $a_k$ is of class $C^3$ on $(-1/2,1/2)\times (1/8,2)$.

\begin{step1} The function $F_{f,g}$ is of class $C^3$. Moreover, for every multi-index $\beta\neq(0,0)$ with $|\beta|\leqslant 3$, we have
\[
\partial^\beta F_{f,g}(\alpha,\eta)
=\sum_{k=1}^{\infty} \partial^\beta a_k(\alpha,\eta),
\qquad (\alpha,\eta)\in(-1/2,1/2)\times(1/8,2).
\]
\end{step1}
\begin{proof}
For any $(\alpha,\eta)\in (-1/2,1/2)\times (1/8,2)$ and for any $k\in \N$, we have
\[\begin{aligned}\vert a_k(\alpha,\eta)\vert &=M\biggl(\frac{\vert f(k)+\alpha g(k)\vert }{\eta}\biggr)\\&\leqslant M\biggl(\frac{\vert  f(k)\vert +1/2\vert g(k)\vert }{\eta}\biggr)\\&\leqslant M\bigl(8\vert  f(k)\vert +4\vert g(k)\vert \bigr).
\end{aligned}\]
Since $\sum_{k} M\bigl(8\vert  f(k)\vert +4\vert g(k)\vert \bigr)<\infty$, it follows that the series $\sum_{k}a_k(\alpha,\eta)$ converges uniformly on $(-1/2,1/2)\times (1/8,2)$. 

We will use the convention that $\operatorname{sgn}(0)=0$. 

For any $(\alpha,\eta)\in (-1/2,1/2)\times (1/8,2)$ and for any $k\in \N$, we have
\[\frac{\partial a_k}{\partial \alpha}(\alpha,\eta)= \frac{g(k)}{\eta}{\rm sgn}\bigl(f(k)+\alpha g(k)\bigr)M'\biggl(\frac{\vert f(k)+\alpha g(k)\vert }{\eta}\biggr).\]   
Hence, since $M'$ is increasing,  \[\begin{aligned}
\biggl\vert \frac{\partial a_k}{\partial \alpha}(\alpha,\eta)\biggr\vert&\leqslant \frac{\vert g(k)\vert}{\eta} M'\biggl(\frac{\vert f(k)+\alpha g(k)\vert }{\eta}\biggr)\\&\leqslant 8 \vert g(k)\vert M'\bigl(8\vert  f(k)\vert +4\vert g(k)\vert \bigr)\\&\leqslant 8 \vert g(k)\vert M'(\vert g(k)\vert)+8 \bigl(8\vert  f(k)\vert +4\vert g(k)\vert \bigr)M'\bigl(8\vert  f(k)\vert +4\vert g(k)\vert \bigr)\\&\leqslant 8KM(\vert g(k)\vert )+8K M\bigl(8\vert  f(k)\vert +4\vert g(k)\vert \bigr).
\end{aligned}\]
It follows that $\sum_{k}\frac{\partial a_k}{\partial \alpha}(\alpha,\eta)$ converges uniformly on $(-1/2,1/2)\times (1/8,2)$. 

For any $(\alpha,\eta)\in (-1/2,1/2)\times (1/8,2)$ and for any $k\in \N$, we have 
\[\frac{\partial a_k}{\partial \eta}(\alpha,\eta)= -\frac{\vert f(k)+\alpha g(k)\vert }{\eta^2}M'\biggl(\frac{\vert f(k)+\alpha g(k)\vert }{\eta}\biggr).\]

Hence, 
\[\begin{aligned}
\biggl\vert \frac{\partial a_k}{\partial \eta}(\alpha,\eta)\biggl\vert&\leqslant\frac{1}{\eta}\frac{\vert f(k)+\alpha g(k)\vert }{\eta} M'\biggl(\frac{\vert f(k)+\alpha g(k)\vert }{\eta}\biggr)\\&\leqslant 8 \bigl(8\vert  f(k)\vert +4\vert g(k)\vert \bigr)M'\bigl(8\vert  f(k)\vert +4\vert g(k)\vert \bigr)\\&\leqslant8K M\bigl(8\vert  f(k)\vert +4\vert g(k)\vert \bigr).
\end{aligned}\]
It follows that $\sum_{k}\frac{\partial a_k}{\partial \eta}(\alpha,\eta)$ converges uniformly on $(-1/2,1/2)\times (1/8,2)$. 

For any $(\alpha,\eta)\in (-1/2,1/2)\times (1/8,2)$ and for any $k\in \N$, we have 
\[\begin{aligned}\frac{\partial^2 a_k}{\partial \alpha\,\partial \eta}(\alpha, \eta)= -& \frac{g(k)}{\eta^2}{\rm sgn}\bigl(f(k)+\alpha g(k)\bigr)M'\biggl(\frac{\vert f(k)+\alpha g(k)\vert }{\eta}\biggr)\\-&\frac{g(k)}{\eta^2}\frac{f(k)+\alpha g(k)}{\eta}M''\biggl(\frac{\vert f(k)+\alpha g(k)\vert }{\eta}\biggr).\end{aligned}\]
Hence, \[\begin{aligned}\biggl\vert\frac{\partial^2 a_k}{\partial \alpha\,\partial \eta}(\alpha, \eta)\biggr\vert&\leqslant 64 \vert g(k)\vert M'\biggl(\frac{\vert f(k)+\alpha g(k)\vert }{\eta}\biggr)+64 K\vert g(k)\vert M'\biggl(\frac{\vert f(k)+\alpha g(k)\vert }{\eta}\biggr)\\&\leqslant64(K+1)\vert g(k)\vert M'\bigl(8\vert f(k)\vert+4\vert g(k)\bigr)\\&\leqslant 64K(K+1)M(\vert g(k)\vert )+64K(K+1)M\bigl(8\vert  f(k)\vert +4\vert g(k)\vert \bigr).
\end{aligned}\]
It follows that $\sum_{k}\frac{\partial^2 a_k}{\partial \alpha\,\partial \eta}(\alpha, \eta)(=\sum_{k}\frac{\partial^2 a_k}{\partial \eta\,\partial \alpha}(\alpha, \eta)$ because $a_k$ is $C^3$ for every $k$) converges uniformly on $(-1/2,1/2)\times (1/8,2)$. 

For any $(\alpha,\eta)\in (-1/2,1/2)\times (1/8,2)$ and for any $k\in \N$, we have 
\[\frac{\partial^2 a_k}{\partial \alpha^2}(\alpha, \eta)=  \frac{\vert g(k)\vert^2}{\eta^2}M''\biggl(\frac{\vert f(k)+\alpha g(k)\vert }{\eta}\biggr).\]
Hence, since $M''$ is increasing,\[\begin{aligned}\biggl\vert\frac{\partial^2 a_k}{\partial \alpha^2}(\alpha, \eta)\biggr\vert&\leqslant 64 \vert g(k)\vert^2 M''(\vert g(k)\vert )+64 \frac{\vert f(k)+\alpha g(k)\vert^2 }{\eta^2} M''\biggl(\frac{\vert f(k)+\alpha g(k)\vert }{\eta}\biggr)\\&\leqslant64K^2M(\vert g(k)\vert)+64 K^2 M\biggl(\frac{\vert f(k)+\alpha g(k)\vert }{\eta}\biggr)\\&\leqslant 64K^2M(\vert g(k)\vert )+64K^2 M\bigl(8\vert  f(k)\vert +4\vert g(k)\vert \bigr).
\end{aligned}\]
It follows that $\sum_{k}\frac{\partial^2 a_k}{\partial \alpha^2}(\alpha, \eta)$ converges uniformly on $(-1/2,1/2)\times (1/8,2)$. 

For any $(\alpha,\eta)\in (-1/2,1/2)\times (1/8,2)$ and for any $k\in \N$, we have 
\[\begin{aligned}\frac{\partial^2 a_k}{\partial \eta^2}(\alpha, \eta)=  2&\frac{\vert f(k)+\alpha g(k)\vert}{\eta^3}M'\biggl(\frac{\vert f(k)+\alpha g(k)\vert }{\eta}\biggr)\\+&\frac{\vert f(k)+\alpha g(k)\vert^2}{\eta^4}M''\biggl(\frac{\vert f(k)+\alpha g(k)\vert }{\eta}\biggr) \end{aligned}\]
Hence, since $M''$ is increasing,\[\begin{aligned}\biggl\vert\frac{\partial^2 a_k}{\partial \eta^2}(\alpha, \eta)\biggr\vert&\leqslant 128 K  M\biggl(\frac{\vert f(k)+\alpha g(k)\vert }{\eta}\biggr)+64 K^2 M\biggl(\frac{\vert f(k)+\alpha g(k)\vert }{\eta}\biggr) \\&\leqslant (128K+64 K^2) M\bigl(8\vert  f(k)\vert +4\vert g(k)\vert \bigr).
\end{aligned}\]
It follows that $\sum_{k}\frac{\partial^2 a_k}{\partial \eta^2}(\alpha, \eta)$ converges uniformly on $(-1/2,1/2)\times (1/8,2)$.

For any $(\alpha,\eta)\in (-1/2,1/2)\times (1/8,2)$ and for any $k\in \N$, we have 
\[\frac{\partial^3 a_k}{\partial \alpha^3}(\alpha, \eta)= \frac{(g(k))^3}{\eta^3} {\rm sgn}\bigl(f(k)+\alpha g(k)\bigr)M^{(3)}\biggl(\frac{\vert f(k)+\alpha g(k)\vert }{\eta}\biggr). \]
Hence, since $M^{(3)}$ is increasing, 
\[\begin{aligned}
    \biggl\vert \frac{\partial^3 a_k}{\partial \alpha^3}(\alpha, \eta)\biggr\vert &\leqslant512 \vert g(k)\vert^3M^{(3)}\biggl(\frac{\vert f(k)+\alpha g(k)\vert }{\eta}\biggr)\\&\leqslant512 \vert g(k)\vert^3M^{(3)}(\vert g(k)\vert )+512 \frac{\vert f(k)+\alpha g(k)\vert^3 }{\eta^3}M^{(3)}\biggl(\frac{\vert f(k)+\alpha g(k)\vert }{\eta}\biggr)\\&\leqslant 512 K^3 M(\vert g(k)\vert )+ 512 K^3 M\biggl(\frac{\vert f(k)+\alpha g(k)\vert }{\eta}\biggr) \\&\leqslant 512 K^3 M(\vert g(k)\vert )+ 512 K^3 M\bigl(8\vert  f(k)\vert +4\vert g(k)\vert \bigr).
\end{aligned}\]
It follows that $\sum_{k}\frac{\partial^3 a_k}{\partial \alpha^3}(\alpha, \eta)$ converges uniformly on $(-1/2,1/2)\times (1/8,2)$.

For any $(\alpha,\eta)\in (-1/2,1/2)\times (1/8,2)$ and for any $k\in \N$, we have 
\[\begin{aligned}\frac{\partial^3 a_k}{\partial \eta\,\partial \alpha^2}(\alpha, \eta)=  -&2\frac{\vert  g(k)\vert^2}{\eta^3}M''\biggl(\frac{\vert f(k)+\alpha g(k)\vert }{\eta}\biggr)\\-&\frac{\vert  g(k)\vert^2}{\eta^4}\vert f(k)+\alpha g(k)\vert M^{(3)}\biggl(\frac{\vert f(k)+\alpha g(k)\vert }{\eta}\biggr). \end{aligned}\]
Hence, since $M^{(3)}$ is increasing, 
\[\begin{aligned}
    \biggl\vert \frac{\partial^3 a_k}{\partial \eta\,\partial \alpha^2}(\alpha, \eta)\biggr\vert &\leqslant 2 \frac{\vert g(k)\vert^2}{\eta^3}M''\biggl(\frac{\vert f(k)+\alpha g(k)\vert }{\eta}\biggr)+\frac{\vert g(k)\vert^2}{\eta^3}K M''\biggl(\frac{\vert f(k)+\alpha g(k)\vert }{\eta}\biggr)\\&\leqslant 512(K+2) \vert g(k)\vert^2M''\biggl(\frac{\vert f(k)+\alpha g(k)\vert }{\eta}\biggr)\\&\leqslant 512(K+2) \biggl(\vert g(k)\vert^2M''(\vert g(k)\vert )+ \frac{\vert f(k)+\alpha g(k)\vert^2 }{\eta^2}M''\biggl(\frac{\vert f(k)+\alpha g(k)\vert }{\eta}\biggr)\biggr) \\&\leqslant 512 K^2(K+2)  M(\vert g(k)\vert )+ 512 K^2(K+2)M\biggl(\frac{\vert f(k)+\alpha g(k)\vert }{\eta}\biggr)\\&\leqslant 512 K^2(K+2)  M(\vert g(k)\vert )+ 512 K^2(K+2) M\bigl(8\vert  f(k)\vert +4\vert g(k)\vert \bigr).
\end{aligned}\]
It follows that $\sum_{k}\frac{\partial^3 a_k}{\partial \eta\,\partial \alpha^2}(\alpha, \eta)$ converges uniformly on $(-1/2,1/2)\times (1/8,2)$.

For any $(\alpha,\eta)\in (-1/2,1/2)\times (1/8,2)$ and for any $k\in \N$, we have 
\[\begin{aligned}\frac{\partial^3 a_k}{\partial \eta^3}(\alpha, \eta)=  
-&6\,\frac{\vert f(k)+\alpha g(k)\vert }{\eta^4}M'\biggl(\frac{\vert f(k)+\alpha g(k)\vert }{\eta}\biggr)
\\-&6\,\frac{\vert f(k)+\alpha g(k)\vert^2 }{\eta^5}M''\biggl(\frac{\vert f(k)+\alpha g(k)\vert }{\eta}\biggr)\\-&\frac{\vert f(k)+\alpha g(k)\vert^3 }{\eta^6}M^{(3)}\biggl(\frac{\vert f(k)+\alpha g(k)\vert }{\eta}\biggr). \end{aligned}\]
Hence, 
\[\begin{aligned}
    \biggl\vert \frac{\partial^3 a_k}{\partial \eta^3}(\alpha, \eta)\biggr\vert &\leqslant \biggl(\frac{6}{\eta^3}K+\frac{6}{\eta^3}K^2+\frac{1}{\eta^3}K^3\biggr)M\biggl(\frac{\vert f(k)+\alpha g(k)\vert }{\eta}\biggr)\\&\leqslant 6\times 512 K^3  M\bigl(8\vert  f(k)\vert +4\vert g(k)\vert \bigr).
\end{aligned}\]
It follows that $\sum_{k}\frac{\partial^3 a_k}{\partial \eta^3}(\alpha, \eta)$ converges uniformly on $(-1/2,1/2)\times (1/8,2)$.

For any $(\alpha,\eta)\in (-1/2,1/2)\times (1/8,2)$ and for any $k\in \N$, we have 
\[\begin{aligned}\frac{\partial^3 a_k}{\partial \alpha\,\partial \eta ^2}(\alpha, \eta)= &\frac{2}{\eta^3} g(k) {\rm sgn}\bigl(f(k)+\alpha g(k)\bigr)M'\biggl(\frac{\vert f(k)+\alpha g(k)\vert }{\eta}\biggr)\\&+\frac{4}{\eta^4} g(k) \bigl( f(k)+\alpha g(k)\bigr)  M''\biggl(\frac{\vert f(k)+\alpha g(k)\vert }{\eta}\biggr)\\& + \frac{1}{\eta^5} g(k) \vert f(k)+\alpha g(k)\vert ^2  {\rm sgn}\bigl(f(k)+\alpha g(k)\bigr) M^{(3)}\biggl(\frac{\vert f(k)+\alpha g(k)\vert }{\eta}\biggr). \end{aligned}\]
Hence, since $M'$ is increasing, 

\[\begin{aligned}
    \biggl\vert \frac{\partial^3 a_k}{\partial \alpha\,\partial \eta ^2}(\alpha, \eta)\biggr\vert& \leqslant \biggl(\frac{2}{\eta^3}+\frac{4}{\eta^3}K+\frac{1}{\eta^3}K^2\biggl)\vert g(k)\vert M'\biggl(\frac{\vert f(k)+\alpha g(k)\vert }{\eta}\biggr)\\&\leqslant 7\times 512 K^2 \vert g(k)\vert M'\biggl(\frac{\vert f(k)+\alpha g(k)\vert }{\eta}\biggr)\\&\leqslant 3584 K^2 \vert g(k)\vert M'(\vert g(k)\vert )+ 3584 K^2 \frac{\vert f(k)+\alpha g(k)\vert }{\eta} M'\biggl(\frac{\vert f(k)+\alpha g(k)\vert }{\eta}\biggr)\\&\leqslant 3584 K^3 M(\vert g(k)\vert )+3584 K^3 M\bigl(8\vert  f(k)\vert +4\vert g(k)\vert \bigr).
\end{aligned}\]
It follows that $\sum_{k}\frac{\partial^3 a_k}{\partial \alpha\,\partial \eta ^2}(\alpha, \eta)$ converges uniformly on $(-1/2,1/2)\times (1/8,2)$.

\smallskip All the above uniformly convergent series on $(-1/2,1/2)\times(1/8,2)$, together with the fact that each $a_k$ is of class $C^3$, complete the proof of Step~1.
\end{proof}
\begin{step2}
 There exists $C^M_0>1$ such that for any $f,g\in\ell_M$ such that $4/5\leqslant\Vert f\Vert,\Vert g\Vert \leqslant5/4$, we have 
    \[\bigl \vert \partial^\beta F_{f,g}(\alpha,\eta)\bigr\vert \leqslant C_0,
\] for any $(\alpha,\eta)\in (-1/2,1/2)\times (1/8,2)$ and for any multi-index $\beta\neq(0,0)$ with $|\beta|\leqslant 3$. 
\end{step2}
\begin{proof}
Let $f,g\in \ell_M $ be such that $4/5<\Vert f\Vert,\Vert g\Vert <5/4$. 

    From the computation made in the proof of Step~1, we see that for any multi-index $\beta\neq(0,0)$ with $|\beta|\leqslant 3$ and for any $(\alpha,\eta)\in (-1/2,1/2)\times (1/8,2)$, we have 
   \[\bigl\vert  \partial^\beta F_{f,g}(\alpha,\eta)\bigr\vert \leqslant 3584 K^3 \biggl(\sum_{k=1}^\infty M(\vert g(k)\vert )+\sum_{k=1}^\infty M(8\vert f(k)\vert +4 \vert g(k)\vert )\biggr).  \]
   Since, $M$ satisfies the $\Delta_2$-condition at zero, il follows that for every $l>1$, there exists $C(l)>1$ such that \begin{equation}\label{Delta2} M(lx)\leqslant C(l) M(x),\qquad \text{for any }\, x\in [0,1].\end{equation}
 Since $\Vert g\Vert\leqslant5/4$, it follows that 
 \[\sum_{k=1}^\infty M\biggl(\frac{\vert g(k)\vert }{5/4}\biggr)\leqslant 1. \]
 So, 
\[\sum_{k=1}^\infty M(\vert g(k)\vert )=\sum_{k=1}^\infty M\biggl(\frac{5}{4}\frac{\vert g(k)\vert }{5/4}\biggr) \leqslant C(5/4).\]
In the same way, since $\Vert f\Vert, \Vert g\Vert \leqslant5/4$, we have that  $\bigl\Vert 8\vert f\vert +4\vert g\vert \bigr\Vert\leqslant15$, and so 
\[\sum_{k=1}^\infty M(8\vert f(k)\vert +4\vert g(k)\vert )\leqslant C(15).\]
Thus the proof of Step~2 is complete by taking
\[
C_0^M = 3584\,K^3\bigl(C(5/4)+C(15)\bigr),
\]
which depends only on $M$.
    
\end{proof}
Now, we  define $N_{f,g}:(-1/2,1/2)\to \R$ by $N_{f,g}(\alpha)=\Vert f+\alpha g\Vert$. Note that, since $M $ satisfies that $\Delta_2$-condition at $0$, we have 
\[\sum_{k=1}^\infty M\biggr(\frac{\vert f(k)+\alpha g(k)\vert}{N_{f,g}(\alpha)}\biggr)=1, \qquad \text{for any } \, \alpha\in (-1/2,1/2).\]
\begin{step3} The function $N_{f,g}$ is of class $C^3$ on $(-1/2,1/2)$.
\end{step3}
\begin{proof}
    It is easy to see that, since $4/5\leqslant\Vert f \Vert, \Vert g\Vert \leqslant5/4$, we have 
    \[
1/8<N_{f,g}(\alpha)<2,
\qquad \text{for any } \,\alpha\in(-1/2,1/2).
\]
Note that for any $t\geqslant0$, we have $tM'(t)\geqslant M(t)$, since 
\[M(t)=\int_{0}^t M'(s)ds \leqslant \int_{0}^t M'(t)ds= tM'(t), \]
because $M(0)=0$ and $M'$ is increasing.

So, for any $\alpha\in(-1/2,1/2)$, we have 
\begin{equation}\label{second}
    \begin{aligned}
   \frac{\partial F_{f,g}}{\partial \eta}(\alpha,N_{f,g}(\alpha))&= -\sum_{k=1}^\infty\frac{\vert f(k)+\alpha g(k)\vert }{\bigl(N_{f,g}(\alpha)\bigr)^2}M'\biggl(\frac{\vert f(k)+\alpha g(k)\vert }{N_{f,g}(\alpha)}\biggr) \\&\leqslant -\frac{1}{N_{f,g}(\alpha)}\sum_{k=1}^\infty M\biggl(\frac{\vert f(k)+\alpha g(k)\vert }{N_{f,g}(\alpha)}\biggr)\\&=-\frac{1}{N_{f,g}(\alpha)}< -\frac{1}{2}<0. 
\end{aligned}\end{equation}
Since $F_{f,g}$ is of class $C^3$ on $(-1/2,1/2)\times (1,8)$, $F_{f,g}(\alpha,N_{f,g}(\alpha))=0$ and  $\frac{\partial F_{f,g}}{\partial \eta}(\alpha,N_{f,g}(\alpha))\neq 0$ for any $\alpha\in (-1/2,1/2)$,  it follows by the implicit function theorem, that  $N_{f,g}$ is of class $C^3$ on  $(-1/2,1/2)$.\end{proof}

\begin{step4}
    There exists $C_M>1$ such that for any $f,g\in\ell_M$ with $4/5\leqslant\Vert f\Vert,\Vert g\Vert \leqslant5/4$, we have for any $\alpha\in (-1/2,1/2)$, 
    
      \[\bigl\vert N_{f,g}(\alpha)-N_{f,g}(0)-\alpha N'_{f,g}(0)-\frac{\alpha^2}{2}N''_{f,g}(0)\bigr\vert \leqslant C\vert\alpha\vert^3. \]
    \end{step4} 
    \begin{proof}

    Let $f,g\in\ell_M$ be such that $4/5\leqslant \Vert f\Vert,\Vert g\Vert \leqslant5/4$. 

    Recall that, by Step 2 and the proof of Step 3, for any $\alpha\in (-1/2,1/2)$, we have 
    \[\bigl\vert \partial_\eta F_{f,g}\bigl(\alpha,N_{f,g}(\alpha)\bigr)\bigl\vert =-\partial_\eta F_{f,g}\bigl(\alpha,N_{f,g}(\alpha)\bigr)>\frac{1}{2},\]
    and that 
      \[\bigl \vert \partial^\beta F_{f,g}\bigl(\alpha,N_{f,g}(\alpha)\bigr)\bigr\vert \leqslant C_0,
\]  for any multi-index $\beta\neq(0,0)$ with $|\beta|\leqslant 3$. 

    For any $\alpha\in (-1/2,1/2)$, we have 
\begin{equation}\label{formula0}
       F_{f,g}\bigl((\alpha,N_{f,g}(\alpha)\bigr)=0. 
\end{equation}
By differentiating \eqref{formula0} with respect to $\alpha$, we obtain 
\begin{equation}\label{formula1}
\partial_\alpha F_{f,g}\bigl(\alpha,N_{f,g}(\alpha)\bigr)
+N'_{f,g}(\alpha)\,\partial_\eta F_{f,g}\bigl(\alpha,N_{f,g}(\alpha)\bigr)=0.
\end{equation}
Hence, 
\begin{equation}\label{formula1'}\bigl\vert N'_{f,g}(\alpha)\bigr\vert \leqslant 2C_0:=C_1 . \end{equation}
Differentiating \eqref{formula1} with respect to \(\alpha\) gives 
\[
\begin{aligned}
0 &= \partial_{\alpha\alpha} F_{f,g}(\alpha,N_{f,g}(\alpha))
     + N_{f,g}'(\alpha)\,\partial_{\alpha\eta} F_{f,g}(\alpha,N_{f,g}(\alpha))\\
  &\qquad + N_{f,g}''(\alpha)\,\partial_\eta F_{f,g}(\alpha,N_{f,g}(\alpha))
     + N_{f,g}'(\alpha)\bigl(\partial_{\alpha\eta} F_{f,g}(\alpha,N_{f,g}(\alpha))
     + N_{f,g}'(\alpha)\,\partial_{\eta\eta} F_{f,g}(\alpha,N_{f,g}(\alpha))\bigr)\end{aligned}\]
     So, \begin{equation}\label{formula2}
\begin{aligned}
0 &= \partial_{\alpha\alpha} F_{f,g}(\alpha,N_{f,g}(\alpha)) \\
  &\quad + 2N_{f,g}'(\alpha)\,\partial_{\alpha\eta} F_{f,g}(\alpha,N_{f,g}(\alpha)) \\
  &\quad + \bigl(N_{f,g}'(\alpha)\bigr)^2\,\partial_{\eta\eta} F_{f,g}(\alpha,N_{f,g}(\alpha)) \\
  &\quad + N_{f,g}''(\alpha)\,\partial_\eta F_{f,g}(\alpha,N_{f,g}(\alpha)).
\end{aligned}
\end{equation}
Hence, 
\[\vert N_{f,g}''(\alpha)\vert \leqslant 2C_0\bigl(1+2C_1+C_1^2\bigr):=C_2.\]
Differentiating \eqref{formula2} with respect to \(\alpha\) gives
\[
\begin{aligned}
0 &
= \partial_{\alpha\alpha\alpha}F_{f,g}\bigl(\alpha,N_{f,g}(\alpha)\bigr)
  + N_{f,g}'(\alpha)\,\partial_{\alpha\alpha\eta}F_{f,g}\bigl(\alpha,N_{f,g}(\alpha)\bigr)\\
&\quad + 2N_{f,g}''(\alpha)\,\partial_{\alpha\eta}F_{f,g}\bigl(\alpha,N_{f,g}(\alpha)\bigr)
  \\ &\quad +2N_{f,g}'(\alpha)\biggl(\partial_{\alpha\alpha\eta}F_{f,g}\bigl(\alpha,N_{f,g}(\alpha)\bigr)+N_{f,g}'(\alpha)\,\partial_{\alpha\eta\eta}F_{f,g}\bigl(\alpha,N_{f,g}(\alpha)\bigr)\biggr)\\
&\quad + 2N_{f,g}'(\alpha)N_{f,g}''(\alpha)\,\partial_{\eta\eta}F_{f,g}\bigl(\alpha,N_{f,g}(\alpha)\bigr)
\\ &\quad + (N_{f,g}'(\alpha))^2\biggl(\partial_{\alpha\eta\eta}F_{f,g}\bigl(\alpha,N_{f,g}(\alpha)\bigr)+N_{f,g}'(\alpha)\,\partial_{\eta\eta\eta}F_{f,g}\bigl(\alpha,N_{f,g}(\alpha)\bigr)\biggr)\\
&\quad + N_{f,g}'''(\alpha)\,\partial_{\eta}F_{f,g}\bigl(\alpha,N_{f,g}(\alpha)\bigr)
\\ &\quad + N_{f,g}''(\alpha)\biggl(\partial_{\alpha\eta}F_{f,g}\bigl(\alpha,N_{f,g}(\alpha)\bigr)+N_{f,g}'(\alpha)\,\partial_{\eta\eta}F_{f,g}\bigl(\alpha,N_{f,g}(\alpha)\bigr)\biggr)
\end{aligned}
\]
So, 
\[\vert N_{f,g}'''(\alpha)\vert \leqslant2C_0\bigl(1+4C_1+3C_2+3C_1^2+2C_1C_2+C_1^3\bigr):= C_3. \]
Therefore,  by Taylor's theorem (Lagrange form of the remainder), for all $\alpha\in(-1/2,1/2)$, we have 
 \[\bigl\vert N_{f,g}(\alpha)-N_{f,g}(0)-\alpha N'_{f,g}(0)-\frac{\alpha^2}{2}N''_{f,g}(0)\bigr\vert \leqslant \frac{C_3}{6}\vert\alpha\vert^3. \]
  \end{proof}
    \begin{step5} For any $\varepsilon >0$, there exists $0<h_M(\epsilon)<1$  such that the following holds: For any $f,g\in \ell_M$  with $4/5\leqslant \Vert f\Vert ,\Vert g\Vert \leqslant 5/4$, if there  do not exist $\tild{f}, \tild{g}\in \ell_M$ satisfying 
\begin{equation}\tag{$\ast$}\label{star}
\left\{
\begin{aligned}
&\operatorname{supp}(\widetilde f)\cap \operatorname{supp}(\widetilde g)=\emptyset,\\
&\|f-\widetilde f\|\leqslant\varepsilon,\\
&\|g-\widetilde g\|\leqslant\varepsilon,
\end{aligned}
\right.
\end{equation}
then \[\frac{\partial^2 F_{f,g}}{\partial\alpha^2}\bigl(0,N_{f,g}(0)\bigr)\geqslant h_M(\varepsilon). \]
     \end{step5}
     \begin{proof}
         We have already seen in the proof of Step 3 that for any $t\geqslant 0$, we have $tM'(t)\geqslant M(t)$. In the same way, we can show that $tM''(t)\geqslant M'(t)$ (because $M'(0)=0$ and $M''$ is increasing). So, for any $t\geqslant 0$, 
         \begin{equation}\label{rk}
             t^2M''(t)\geqslant M(t). 
         \end{equation}
     
         Let $\varepsilon >0$ and let $f,g\in \ell_M$  with $4/5\leqslant \Vert f\Vert ,\Vert g\Vert \leqslant5/4$ be such that there do not exist $\widetilde f,\widetilde g\in\ell_M$ satisfying \eqref{star}. 
         
         Let \[A:=\{k\in \N; \; \vert f(k)\vert \geqslant \vert g(k)\vert\} \qquad \text{and}\qquad  B:=\{k\in \N; \; \vert f(k)\vert <\vert g(k)\vert \}.\] 

         We set \[\widetilde{f}:=f \ind_A \qquad \text{and}\qquad \widetilde{g}:=g \ind_B .\]

         Since $\operatorname{supp}(\widetilde f)\cap \operatorname{supp}(\widetilde g)=\emptyset$, it follows that 
         \[\bigr\Vert f-\widetilde{f}\bigr\Vert>\varepsilon \qquad\text{or} \qquad \bigr\Vert g-\widetilde{g}\bigr\Vert>\varepsilon.   \]
         \smallskip  \textbf{Case 1:} $\bigr\Vert g-\widetilde{g}\bigr\Vert>\varepsilon$. 

         By \eqref{rk}, \eqref{Delta2}, and since $M''$ is increasing, we have 
         \[\begin{aligned}
            \frac{\partial^2 F_{f,g}}{\partial\alpha^2}\bigl(0,N_{f,g}(0)\bigr)&= \sum_{k=1}^\infty \frac{\vert g(k)\vert^2}{(N_{f,g}(0)) ^2}M''\biggl(\frac{\vert f(k)\vert }{N_{f,g}(0)}\biggr)\\&=\sum_{k=1}^\infty \frac{\vert g(k)\vert^2}{\Vert f\Vert ^2}M''\biggl(\frac{\vert f(k)\vert }{\Vert f\Vert }\biggr)\\&\geqslant \sum_{k=1}^\infty \frac{\vert g(k)\vert^2}{\Vert f\Vert ^2}M''\biggl(\frac{\vert f(k)\vert }{\Vert f\Vert }\biggr) \ind_A\\& \geqslant \sum_{k=1}^\infty \frac{\vert g(k)\vert^2}{\Vert f\Vert ^2}M''\biggl(\frac{\vert g(k)\vert }{\Vert f\Vert }\biggr) \ind_A \\&=\sum_{k=1}^\infty  \frac{\vert (g-\widetilde g)(k)\vert ^2}{\Vert f\Vert ^2} M''\biggl(\frac{\vert (g-\widetilde g)(k)\vert }{\Vert f\Vert }\biggr)\\&\geqslant \sum_{k=1}^\infty  M\biggl(\frac{\vert (g-\widetilde g)(k)\vert }{\Vert f\Vert }\biggr) \\&= \sum_{k=1}^\infty  M\biggl(\frac{\vert (g-\widetilde g)(k)\vert }{\Vert g-\widetilde g\Vert }\frac{\Vert g-\widetilde g\Vert }{\Vert f\Vert }\biggr)\\&\geqslant \sum_{k=1}^\infty  M\biggl(\frac{4\epsilon}{5}\frac{\vert (g-\widetilde g)(k)\vert }{\Vert g-\widetilde g\Vert }\biggr)\\&\geqslant\frac{1}{C(\frac{5}{4\epsilon})}\sum_{k=1}^\infty  M\biggl(\frac{\vert (g-\widetilde g)(k)\vert }{\Vert g-\widetilde g\Vert }\biggr)\\&= \frac{1}{C(\frac{5}{4\epsilon})}:=h_M(\varepsilon).
            \end{aligned}\]
            \smallskip  \textbf{Case 2:} $\bigr\Vert f-\widetilde{f}\bigr\Vert>\varepsilon$. 
            
                In the same way, we show that $\frac{\partial^2 F_{f,g}}{\partial\alpha^2}\bigl(0,N_{f,g}(0)\bigr)\geqslant h_M(\varepsilon)$. 
     \end{proof}

The following last step concludes the proof of Theorem \ref{thm:MainDisjointnessPreservation}.
     
   \begin{step6} Let $M_1$ and $M_2$ be good Orlicz functions.
   For any $\varepsilon >0$, there exists $\delta>0$ such that  the following holds:
   For every disjoint $u, v \in \ell_{M_1}$ with $\|u\|_{M_1} = \|v\|_{M_1} = 1$, and every $f,g\in \ell_{M_2}$  with $\frac{1}{1+\delta}\leqslant \Vert f\Vert_{M_2} ,\Vert g\Vert_{M_2} \leqslant 1+\delta$, if there  do not exist $\tild{f}, \tild{g}\in \ell_{M_2}$ satisfying 
\begin{equation}\tag{$\ast$}\label{star2}
\left\{
\begin{aligned}
&\operatorname{supp}(\widetilde f)\cap \operatorname{supp}(\widetilde g)=\emptyset,\\
&\|f-\widetilde f\|_{M_2}\leqslant\varepsilon,\\
&\|g-\widetilde g\|_{M_2}\leqslant\varepsilon,
\end{aligned}
\right.
\end{equation}
then there exists $\alpha\in \R$ such that 
\[\Vert f+\alpha g \Vert_{M_2} > (1+\delta) \Vert u+\alpha v\Vert_{M_1}. \]
   \end{step6}
   \begin{proof}
Let 
\[C:=\max\{C_{M_1}, C_{M_2}\},\]
where $C_{M_1}$ and  $C_{M_2}$ are given by Step 4. So, for every $\alpha\in (-1/2,1/2)$, we have 
\[N_{u,v}(\alpha)-1-\alpha N'_{u,v}(0)-\alpha^2/2 N''_{u,v}(0)\leqslant C \vert \alpha\vert ^3. \]
Since $N_{u,v}$ has a minimum at $0$ (because $u$ and $v$ are disjointly supported), it follows that $N'_{u,v}(0)=0$. 

Moreover, since 
\[\frac{\partial^2F_{u,v}}{\partial\alpha^2}(0,N_{u,v}(0))=\sum_{k=1}^\infty \vert v(k)\vert ^2M'' (\vert u(k)\vert )=0,\]
and 
\[\frac{\partial F_{u,v}}{\partial\eta}(0,N_{u,v}(0))\neq 0\]
(by \eqref{second}), it follows from \eqref{formula2} that \[N''_{u,v}(0)=0.\]
Therefore, for any $\alpha\in(-1/2,1/2)$, we have 
\begin{equation}\label{T1}
N_{u,v}(\alpha)\leqslant 1+C\vert \alpha\vert^3. 
\end{equation}
Fix $\varepsilon>0$. Define
\[
h_1(\varepsilon)=\frac{h_{M_2}(\varepsilon)}{6C_0^{M_2}},
\]
where $C_0^{M_2}$ and $h_{M_2}(\varepsilon)$ are given in Steps~2 and~5, respectively.

Set
\[
\alpha_0=\frac{h_1(\varepsilon)}{8C}.
\] Note that $0<\alpha_0<\tfrac18$.
Finally put
\[
\delta_1=\frac{1}{1-\dfrac{\alpha_0^2 h_1(\varepsilon)}{4}}-1,\qquad
\delta_2=\frac{1+\dfrac{\alpha_0^2 h_1(\varepsilon)}{4}}{1+\dfrac{\alpha_0^2 h_1(\varepsilon)}{8}}-1,
\]
and
\[
\delta=\min\Bigl\{\tfrac14,\ \delta_1,\ \delta_2\Bigr\}.
\]
Let $f,g \in \ell_{M_2}$  with $\frac{1}{1+\delta}\leqslant \Vert f\Vert ,\Vert g\Vert \leqslant 1+\delta$ be such that there  do not exist $\tild{f}, \tild{g}\in \ell_{M_2}$ satisfying \eqref{star2}. 

Since $0<\delta\leqslant 1/4$, it follows that $4/5  \leqslant \Vert f\Vert ,\Vert g\Vert \leqslant 5/4$. Note  that  $N''_{f,g}(0)\geqslant 0$ since $N_{f,g}$ is convex. 

We claim that
\begin{equation}\label{fact}
\vert N'_{f,g}(0)\vert\leqslant h_1(\varepsilon)\quad\Longrightarrow\quad
N''_{f,g}(0)>2h_1(\varepsilon).
\end{equation}
Indeed, if $\vert N'_{f,g}(0)\vert\leqslant h_1(\varepsilon)$ and $N''_{f,g}(0)\leqslant 2h_1(\varepsilon)$, then \eqref{formula2} together with Step 2 imply that 
\[ \begin{aligned} \frac{\partial^2F_{f,g}}{\partial \alpha^2}(0,N_{f,g}(0))&\leqslant 2C_0^{M_2}h_1(\varepsilon)+C_0^{M_2}(h_1(\varepsilon))^2+2C_0^{M_2}h_1(\varepsilon)\\&<4C_0^{M_2}h_1(\varepsilon)+C_0^{M_2}h_1(\varepsilon)=\frac{5}{6}h_{M_2}(\varepsilon)<h_{M_2}(\varepsilon),\end{aligned}\]
which  contradicts Step~5, and hence proves \eqref{fact}.

\medskip
Our aim now is to show that 
\[N_{f,g}(\alpha_0)>(1+\delta)N_{u,v}(\alpha_0) \qquad\text{or}\qquad N_{f,g}(-\alpha_0)>(1+\delta)N_{u,v}(-\alpha_0),\]
which will complete the proof of Step~6.

\smallskip We distinguish two cases.

 \smallskip  \textbf{Case 1:} $N'_{f,g}(0)\geqslant 0$. 

 By \eqref{fact}, there are two subcases: either $N'_{f,g}(0)>h_1(\varepsilon)$ or $N''_{f,g}(0)>2h_1(\varepsilon)$.
 
 - If $N'_{f,g}(0)>h_1(\varepsilon)$, then by Step 4, we have for every $0< \alpha<1/2$, \[\begin{aligned}
     N_{f,g}(\alpha)&\geqslant N_{f,g}(0)+\alpha N'_{f,g}(0)+\frac{\alpha^2}{2} N''_{f,g}(0)-C\alpha^3\\&\geqslant \Vert f\Vert +\alpha N'_{f,g}(0)-C\alpha^3\\&>\Vert f\Vert +\alpha h_1(\varepsilon)-C\alpha^3 \\&\geqslant  \Vert f\Vert +\alpha^2 h_1(\varepsilon)-C\alpha^3 \\&=\Vert f \Vert + \alpha^2\bigl(h_1(\varepsilon)-C\alpha\bigr). 
 \end{aligned}\]
 
 - If $N''_{f,g}(0)>2h_1(\varepsilon)$, then for every $0< \alpha<1/2$, \[\begin{aligned}
     N_{f,g}(\alpha)&\geqslant N_{f,g}(0)+\alpha N'_{f,g}(0)+\frac{\alpha^2}{2} N''_{f,g}(0)-C\alpha^3\\&\geqslant \Vert f\Vert +\frac{\alpha^2}{2} N''_{f,g}(0)-C\alpha^3\\&>\Vert f\Vert +\alpha^2 h_1(\varepsilon)-C\alpha^3  \\&=\Vert f \Vert + \alpha^2\bigl(h_1(\varepsilon)-C\alpha\bigr). 
 \end{aligned}\]

Hence, in both subcases, 
\[\begin{aligned}
   N_{f,g}(\alpha_0)&> \Vert f \Vert + \alpha_0^2\bigl(h_1(\varepsilon)-C\alpha_0\bigr)\\&=\Vert f\Vert +\frac{7}{8}\alpha_0^2 h_1(\varepsilon)\\&>\Vert f\Vert +\frac{1}{2}\alpha_0^2h_1(\varepsilon).
\end{aligned}\]
Therefore, by \eqref{T1},
\[\begin{aligned}\frac{N_{f,g}(\alpha_0)}{N_{u,v}(\alpha_0)}&>\frac{\Vert f\Vert +\frac{1}{2}\alpha_0^2h_1(\varepsilon)}{1+\frac{1}{8}\alpha_0^2h_1(\varepsilon)}\\& \geqslant\frac{\frac{1}{1+\delta} +\frac{1}{2}\alpha_0^2h_1(\varepsilon)}{1+\frac{1}{8}\alpha_0^2h_1(\varepsilon)}\\& \geqslant\frac{\frac{1}{1+\delta_1} +\frac{1}{2}\alpha_0^2h_1(\varepsilon)}{1+\frac{1}{8}\alpha_0^2h_1(\varepsilon)}\\&=\frac{1-\frac{1}{4}\alpha_0^2h_1(\varepsilon) +\frac{1}{2}\alpha_0^2h_1(\varepsilon)}{1+\frac{1}{8}\alpha_0^2h_1(\varepsilon)}\\&=\frac{1+\frac{1}{4}\alpha_0^2h_1(\varepsilon)}{1+\frac{1}{8}\alpha_0^2h_1(\varepsilon)}\\&=1+\delta_2\geqslant 1+\delta,
\end{aligned}\] which is the desired result. 

 \smallskip  \textbf{Case 2:} $N'_{f,g}(0)< 0$. 

 By \eqref{fact}, there are two subcases: either $N'_{f,g}(0)<-h_1(\varepsilon)$ or $N''_{f,g}(0)>2h_1(\varepsilon)$.
 
 - If $N'_{f,g}(0)<-h_1(\varepsilon)$, then by Step 4, we have for every $-1/2< \alpha<0$, \[\begin{aligned}
     N_{f,g}(\alpha)&\geqslant \Vert f\Vert +\alpha N'_{f,g}(0)+\frac{\alpha^2}{2} N''_{f,g}(0)-C\vert\alpha\vert^3\\&\geqslant \Vert f\Vert +\alpha N'_{f,g}(0)-C\vert\alpha\vert^3\\&>\Vert f\Vert -\alpha h_1(\varepsilon)-C\vert\alpha\vert^3 \\&>  \Vert f\Vert +\alpha^2 h_1(\varepsilon)-C\vert\alpha\vert^3  \\&=\Vert f \Vert + \alpha^2\bigl(h_1(\varepsilon)-C\vert\alpha\vert\bigr). 
 \end{aligned}\]

- If $N''_{f,g}(0)>2h_1(\varepsilon)$, then for every $-1/2< \alpha<0$, we have 
\[\begin{aligned}N_{f,g}(\alpha)&\geqslant \Vert f\Vert +\alpha N'_{f,g}(0)+\frac{\alpha^2}{2} N''_{f,g}(0)-C\vert\alpha\vert^3\\&\geqslant\Vert f\Vert +\frac{\alpha^2}{2} N''_{f,g}(0)-C\vert\alpha\vert^3 \\&> \Vert f\Vert +\alpha^2h_1(\varepsilon)-C\vert\alpha\vert^3 \\&=\Vert f\Vert +\alpha^2\bigl(h_1(\varepsilon)-C\vert \alpha\vert).\end{aligned} \]

Hence, in both subcases, 
\[\begin{aligned}
   N_{f,g}(-\alpha_0)&> \Vert f \Vert + \alpha_0^2\bigl(h_1(\varepsilon)-C\alpha_0\bigr)\\&>\Vert f\Vert +\frac{1}{2}\alpha_0^2h_1(\varepsilon).
\end{aligned}\]
Moreover, by \eqref{T1}, 
\[N_{u,v}(-\alpha_0)\leqslant 1+\frac{1}{8}\alpha_0^2h_1(\varepsilon). \]
So, we can show, by the same computation as in Case~1, that \[\frac{N_{f,g}(-\alpha_0)}{N_{u,v}(-\alpha_0)}>1+\delta.\]
   \end{proof}

\bigskip

   \section{Basis preservation results}\label{sec:3}

   The goal of this section is to prove Theorem \ref{thm:MainBasisPreservation}. To this end, we fix an Orlicz function $M$ satisfying its assumptions. Observe that for every $0 \leqslant t, u \leqslant 1$, it holds in particular that $M(tu) \geqslant M(t)M(u)$. Also observe that the condition $M(1) = 1$ implies that all basic vectors have norm $1$.

\begin{fact}\label{fact1}
    Let $\varepsilon>0$. There exists $h(\varepsilon)>0$ such that for any $x\in \ell_M$ such that $\Vert x\Vert \leqslant1$ and any $i_0 \in \N$, 
    \[\vert x(i_0)\vert >1-h(\varepsilon)\implies \Vert x-\operatorname{sgn}(x(i_0))e_{i_0}\Vert <\varepsilon.\]
\end{fact}
\begin{proof}
    Let $\varepsilon>0$ and $C(1/\varepsilon)>1$  be such that \[ M\biggl(\frac{1}{\varepsilon}x\biggl)\leqslant C(1/\varepsilon) M(x),\qquad \text{for any }\, x\in [0,1].\]
   Note that the function $f:t \longmapsto M(t)+1-M(1-t)$ is continuous and strictly increasing on $[0, 1]$. Moreover, $\lim_{t\to 0}f(t)=0$. So one can find $h(\varepsilon)>0$ such that for any $0\leqslant t< h(\varepsilon)$, we have \[f(t)<\frac{1}{C(1/\varepsilon)}.\] 
   Now, let $i_0\in\N$ and $x\in \ell_M$ be such that $\Vert x\Vert \leqslant1$ and $\vert x(i_0)\vert >1-h(\varepsilon)$. Let $\theta=\operatorname{sgn}(x(i_0))$. 

   We have  \[\begin{aligned}
       \sum_{k=1}^\infty M(\vert (x-\theta e_{i_0})(k)\vert)&= M(\vert x(i_0)-\theta\vert)+ \sum_{\substack{k= 1 \\ k\neq i_0}}^\infty  M(\vert x(k)\vert)\\&\leqslant M(1-\vert x(i_0)\vert )+1-M(\vert x(i_0)\vert )\\&= f(1-\vert x(i_0)\vert)<\frac{1}{C(1/\varepsilon)},
   \end{aligned}\]
   since  $1-\vert x(i_0)\vert <h(\varepsilon)$. 

   Hence, \[\sum_{k=1}^\infty M\biggl(\frac{\vert (x-\theta e_{i_0})(k)\vert }{\varepsilon}\biggr)\leqslant C(1/\varepsilon) \sum_{k=1}^\infty M(\vert (x-\theta e_{i_0})(k)\vert)<1,
   \]
   which implies that \[\Vert x-\theta e_{i_0}\Vert <\varepsilon.\]
\end{proof}
\begin{lemma}\label{lemma}
    For any $\varepsilon>0$, there exists $\delta=\delta(\varepsilon)>0$ such that  such that for every  $T\in \operatorname{Emb}_{1+\delta}(\ell_M^2,\ell_M)$, there exist $ \widetilde f,  \widetilde g\in S_{\ell_M}$, such that
    \begin{equation}\tag{$\ast\ast$}\label{2stars}
\left\{
\begin{aligned}
&\operatorname{supp}(\widetilde f)\cap\operatorname{supp}(\widetilde g)=\emptyset,\\
&\Vert T(e_1)- \widetilde f\Vert \leqslant \varepsilon,\\
&\Vert T(e_2)- \widetilde g\Vert \leqslant \varepsilon,\\
&\text{either }\widetilde f=\pm e_{i_0} \text{ or } \widetilde g=\pm e_{i_0}, \text{ for some } i_0\in \N. 
\end{aligned}
\right.
\end{equation}
\end{lemma}
\begin{proof}
    Denote \[r:=\Vert e_1-e_2\Vert. \]
    So, \[ M(1/r)+M(1/r)=1.\]
    Hence,\begin{equation}\label{r}
        M(1/r)=1/2 \qquad \text{and} \qquad r>1. 
    \end{equation}
    \smallskip 
    
   Let $\varepsilon>0$ and let $h(\varepsilon/2)>0$ be the constant from Fact~\ref{fact1}, chosen so that \[h(\varepsilon/2)\leqslant  1-\tfrac1r\cdot\] 
Since $\alpha(h(\varepsilon/2))>1$ (where $\alpha$ is the function given in the statement of Theorem \ref{thm:MainBasisPreservation}), 
 one may fix $0<\varepsilon'<\varepsilon$ such that 
\[
M\biggl(\frac{1}{1+2\varepsilon'} \biggr) > 
\frac{1}{\alpha(h(\varepsilon/2))}\cdot
\]
By Theorem \ref{thm:MainDisjointnessPreservation}, one can fix $0<\delta <\varepsilon'/6$ such that for every $T\in \operatorname{Emb}_{1+\delta}(\ell_M^2,\ell_M)$, there exist $f, g\in \ell_M$ such that 
\begin{equation*}
\left\{
\begin{aligned}
&\operatorname{supp}( f)\cap \operatorname{supp}( g)=\emptyset,\\
&\|T(e_1)- f\|\leqslant\tfrac{\varepsilon'}{6},\\
&\|T(e_2)- g\|\leqslant\tfrac{\varepsilon'}{6}\cdot  
\end{aligned}
\right.
\end{equation*}   
Let \[f_1:= \frac{f}{\Vert f\Vert }\in S_{\ell_M}\] 
and 
\[g_1:= \frac{g}{\Vert g\Vert }\in S_{\ell_M}. \]
Hence, 
\[\operatorname{supp}( f_1)\cap \operatorname{supp}( g_1)=\emptyset.\]
Moreover, \begin{equation}\label{norm} \begin{aligned}
\Vert T(e_1)-f_1\Vert &\leqslant\Vert T(e_1)-f\Vert +\biggl\Vert f-\frac{f}{\Vert f\Vert}\biggr \Vert \\&\leqslant \varepsilon'/6+ \bigl\vert 1-\Vert f\Vert \bigr\vert. 
\end{aligned} \end{equation}
However, $|1-\|f\|| = |1 - \|T(e_1)\|| + \|T(e_1) - f\| \leqslant \delta + \varepsilon'/6 \leqslant \varepsilon'/3$,
and so \eqref{norm} implies that 
\begin{equation}\label{norm f1}
   \Vert T(e_1)-f_1\Vert\leqslant \varepsilon'/2<\varepsilon/2.\end{equation}
In the same way, we show that 
 \begin{equation}\label{norm g1}
   \Vert T(e_2)-g_1\Vert\leqslant \varepsilon'/2<\varepsilon/2.\end{equation}
Therefore, 
\begin{equation}\label{contradiction} 
\Vert f_1-g_1\Vert \leqslant \Vert f_1-T(e_1)\Vert +\Vert T(e_1)-T(e_2)\Vert+\Vert T(e_2)-g_1\Vert\leqslant \varepsilon'+(1+\delta)r. 
\end{equation}

\smallskip
Assume, for the sake of contradiction, that for any $k\in \N$, \[\vert f_1(k)\vert \leqslant 1-h(\varepsilon/2) \qquad\text{and} \qquad\vert g_1(k)\vert \leqslant 1-h(\varepsilon/2).\]
Hence, 
\[\begin{aligned}
    \sum_{k=1}^\infty M\biggl(\frac{\vert f_1(k)-g_1(k)\vert }{ \varepsilon'+(1+\delta)r}\biggr)&\geqslant \sum_{k=1}^\infty M\biggl(\frac{\vert f_1(k)-g_1(k)\vert }{ (1+\delta+\varepsilon')r}\biggr)\\&>\sum_{k=1}^\infty M\biggl(\frac{\vert f_1(k)-g_1(k)\vert }{ (1+2\varepsilon')r}\biggr)\\&=\sum_{k=1}^\infty M\biggl(\frac{\vert f_1(k)\vert }{ (1+2\varepsilon')r}\biggr)+\sum_{k=1}^\infty M\biggl(\frac{\vert g_1(k)\vert }{ (1+2\varepsilon')r}\biggr)\\&\geqslant M\biggl(\frac{1}{1+2\varepsilon'}\biggl)\biggl[\sum_{k=1}^\infty M\Biggl(\frac{\vert f_1(k)\vert }{ r}\biggr)+\sum_{k=1}^\infty M\biggl(\frac{\vert g_1(k)\vert }{r}\biggr)\biggr] \\&> 
\frac{1}{\alpha(h(\varepsilon/2))}\biggl[\alpha(h(\varepsilon/2))M(1/r)\biggl(\sum_{k=1}^\infty M(\vert f_1(k)\vert)+ \sum_{k=1}^\infty M(\vert g_1(k)\vert)\biggr)\Biggr]\\&=1,
\end{aligned}\]
since $M(1/r)=1/2$. 

Therefore, $\Vert f_1-g_1\Vert > \varepsilon'+(1+\delta)r$, which  contradicts \eqref{contradiction}. 

So, there exists $i_0\in \N$ such that 
\[\vert f_1(i_0)\vert > 1-h(\varepsilon/2) \qquad\text{or} \qquad\vert g_1(i_0)\vert > 1-h(\varepsilon/2).\]

Suppose for example that $\vert f_1(i_0)\vert > 1-h(\varepsilon/2)$. So, by Fact \ref{fact1}, 
\begin{equation}\label{norm f12}\Vert f_1-\operatorname{sgn}(f_1(i_0))e_{i_0}\Vert <\varepsilon/2.\end{equation}
So, by taking $\widetilde f:= \operatorname{sgn}(f_1(i_0))e_{i_0}$ and $\widetilde g:=g_1$, we get: 
\begin{enumerate}
    \item $\operatorname{supp}(\widetilde f)\cap \operatorname{supp}(\widetilde g)=\emptyset$ because $i_0\in \operatorname{supp}(f_1)\minus \operatorname{supp}(g_1)$.
    \item By \eqref{norm f1} and \eqref{norm f12}, we have \[\Vert T(e_1)-\widetilde f \Vert\leqslant \Vert T(e_1)-f_1\Vert +\Vert f_1-\widetilde f\Vert<\varepsilon. \]
    \item By \eqref{norm g1},\[\Vert T(e_2)-\widetilde g\Vert =    \Vert T(e_2)-g_1\Vert\leqslant \varepsilon'/2<\varepsilon/2.\]
\end{enumerate} \end{proof}

\begin{proof}[Proof of Theorem \ref{thm:MainBasisPreservation}]
We prove the theorem in the special case $X = \ell_M^2$. The general case immediately follows.
As in the proof of Lemma \ref{lemma}, we denote \[r:=\Vert e_1-e_2\Vert \]
and so \[M(1/r)=1/2. \]

\smallskip
Let $\varepsilon>0$ and let $h(\varepsilon/2)>0$ be the constant from Fact~\ref{fact1}, chosen so that \[h(\varepsilon/2)\leqslant  1-\tfrac1r\cdot\] 

Since $\alpha(h(\varepsilon/2))>1$, one may fix $0<\varepsilon'<\varepsilon$ such that 
\[M\biggl(\frac{1}{1+2\varepsilon'}\biggr)>\frac{2}{1+\alpha(h(\varepsilon/2))}.\]

Let $\delta=\delta(\varepsilon'/2)>0$ be the constant from Lemma~\ref{lemma}, chosen so that $\delta<\varepsilon'$. 

Let $T\in \operatorname{Emb}_{1+\delta}(\ell_M^2,\ell_M)$. By Lemma \ref{lemma}, we suppose for example that 
\[\Vert T(e_1)-\theta_1 e_{i_1}\Vert \leqslant \varepsilon'/2<\varepsilon/2\] and \begin{equation}\label{dernier_2}\Vert T(e_2)-\widetilde g \Vert \leqslant \varepsilon'/2<\varepsilon/2,\end{equation}
where $\widetilde g \in S_{\ell_M}$, $i_1\notin \operatorname{supp}(\widetilde g)$ and $\theta_1\in \{1,-1\}$. 

Hence, 
\begin{equation}\label{normcontradiction}\Vert \widetilde g -\theta_1 e_{i_1}\Vert \leqslant\Vert \widetilde g -T(e_2)\Vert+\Vert T(e_2)-T(e_1)\Vert +\Vert T(e_1)-\theta_1e_{i_1}\Vert \leqslant \varepsilon'+(1+\delta)r. \end{equation}

Suppose, by contradiction, that for every $k\in \N$, we have \[\vert \widetilde g (k)\vert \leqslant 1-h(\varepsilon/2). \] Hence,
\[\begin{aligned}
    \sum_{k=1}^\infty M\biggl(\frac{\vert \widetilde g(k)-\theta_1 e_{i_1}(k)\vert }{ \varepsilon'+(1+\delta)r}\biggr)&\geqslant \sum_{k=1}^\infty M\biggl(\frac{\vert \widetilde g(k)-\theta_1 e_{i_1}(k)\vert }{ (1+2\varepsilon')r}\biggr)\\&=\sum_{k=1}^\infty M\biggl(\frac{\vert \widetilde g(k)\vert }{ (1+2\varepsilon')r}\biggr)+M\biggl(\frac{1}{(1+2\varepsilon')r}\biggr)\\&\geqslant M\biggl(\frac{1}{(1+2\varepsilon')}\biggr)\Biggr[\sum_{k=1}^\infty M\biggl(\frac{\vert \widetilde g(k)\vert }{ r}\biggr)+M\biggl(\frac{1}{r}\biggr)\Biggr]\\&\geqslant M\biggl(\frac{1}{(1+2\varepsilon')}\biggr)\Biggr[\alpha(h(\varepsilon/2))M(1/r)\sum_{k=1}^\infty M(\vert \widetilde g(k)\vert  )+M\biggl(\frac{1}{r}\biggr)\Biggr]\\&=M\biggl(\frac{1}{(1+2\varepsilon')}\biggr)\Biggr[\frac{1}{2}\alpha(h(\varepsilon/2))+\frac{1}{2}\Biggr]>1, 
\end{aligned}\]
which contradicts \eqref{normcontradiction}
Hence, there exists $i_2\in N$, such that \[\vert \widetilde g (i_2)\vert > 1-h(\varepsilon/2).\]
Note that $i_2\neq i_1$ since $i_1\notin \operatorname{supp}(\widetilde g)$.

By Fact \ref{fact1}, we have \begin{equation}\label{dernier}
    \Vert \widetilde g-\theta_2 e_{i_2}\Vert <\varepsilon/2,\end{equation}
    where $\theta_2=\operatorname{sgn}(\widetilde g (i_2))$. 

Therefore, by \eqref{dernier_2} and \eqref{dernier},  we have 
\[\Vert T(e_2)-\theta_2e_{i_2}\Vert \leqslant  \Vert T(e_2)-\widetilde g\Vert + \Vert \widetilde g-\theta_2 e_{i_2}\Vert<\varepsilon.\qedhere\]
\end{proof}

\bigskip

\section{Cofinally Fra\"iss\'e Orlicz sequence spaces}\label{sec:4}

This section is devoted to the proof of Corollary \ref{cor:OrliczFraisse} and Theorem \ref{thm:Counterexample}. We start with introducing the definition of cofinally Fra\"iss\'e Banach spaces. In what follows, if $E$ and $F$ are two finite-dimensional subspaces of the same Banach space, and $\varepsilon > 0$, we write $E \subseteq_\varepsilon F$ if every $x \in \Sphere_E$ is at distance at most $\varepsilon$ of $\Sphere_F$ (here, $S_E$ denotes the unit sphere of $E$).

\begin{definition}\label{def:CofFraisse}
     Let $X$ be a separable Banach space.
    \begin{itemize}
        \item A class $\mathcal{F}$ of finite-dimensional subspaces of $X$ is said to be \textit{cofinal} in $X$ if for every finite-dimensional subspace $E \subseteq X$ and every $\varepsilon > 0$, there exists $F \in \mathcal{F}$ such that $E \subseteq_\varepsilon F$.
        \item The space $X$ is said to be \textit{cofinally Fra\"iss\'e} if for every $\varepsilon>0$, there is a cofinal class $\mathcal{F}$ of finite-dimensional subspaces of $X$ satisfying the following property: for every $E \in \mathcal{F}$, there exists $\delta > 0$ such that $\Iso(X)$ acts $\varepsilon$-transitively on  $\Emb_{1+\delta}(E,X)$.
    \end{itemize}
\end{definition}

As observed in \cite[Proposition 8.15]{CdRD}, cofinally Fra\"iss\'e Banach spaces are guarded Fra\"iss\'e. It is actually not known whether there exist guarded Fra\"iss\'e Banach spaces that are not cofinally Fra\"iss\'e (although such examples exist in the setting of countable first-order structures, see \cite{KKKP}).

\smallskip

Note that if the Orlicz function $M$ satisfies the assumptions of Theorem \ref{thm:MainBasisPreservation}, then by Corollary \ref{cor:BasisPreservation}, the elements of $\Iso(\ell_M)$ are exactly given by signed permutations of the canonical basis. In particular, there is some $\varepsilon > 0$ for which the action of $\Iso(\ell_M)$ on $S_{\ell_M}$ is not $\varepsilon$-transitive. If follows that $\ell_M$ is not weak Fra\"iss\'e. However, it is cofinally Fra\"iss\'e, as we will now show.

\begin{proof}[Proof of Corollary \ref{cor:OrliczFraisse}]
    For every $n \in \N$, let $E_n \coloneq \Span\{e_k \mid 1\leqslant k \leqslant n\} \subseteq \ell_M$. The class $\mathcal{F} \coloneq \{E_n \mid n \in \N\}$ is clearly cofinal in $\ell_M$. So to conclude, it is enough to show that for every $\varepsilon > 0$ and $n \in \N$, there exists $\delta > 0$ for which $\Iso(\ell_M)$ acts $\varepsilon$-transitively on  $\Emb_{1+\delta}(E_n,\ell_M)$.

    \smallskip

    So fix $n \in \N$ and $\varepsilon > 0$. By Theorem \ref{thm:MainBasisPreservation}, there exists $\delta > 0$ such that for every $T \in \Emb_{1+\delta}(E_n, \ell_M)$ and every basic vector $e \in E_n$, there is a basic vector $f \in \ell_M$ and $\theta \in \{-1, 1\}$ for which $\|T(e) - f\| \leqslant \varepsilon/2n$. Now let $T_1, T_2 \in \Emb_{1+\delta}(E_n, \ell_M)$. For every $1 \leqslant k \leqslant n$ and $i \in \{1, 2\}$, fix $\sigma_i(k) \in \N$ and $\theta_k^i \in \{-1, 1\}$ with $\|T_i(e_k) - \theta_k^ie_{\sigma_i(k)}\| \leqslant \varepsilon/2n$. Since, for $k \neq l$, $\|T_i(e_k) - T_i(e_l)\| \geqslant (1 + \delta)^{-1}$, assuming $\varepsilon$ and $\delta$ were chosen small enough, we have $\sigma_i(k) \neq \sigma_i(l)$. Hence, there is a bijection $\tau \colon \N \to \N$ such that $\tau(\sigma_1(k)) = \sigma_2(k)$ for every $1\leqslant k \leqslant n$. Now let $\alpha_{\sigma_1(k)} \coloneq \theta_k^1\theta_k^2$ for $1 \leqslant k \leqslant n$, and $\alpha_{\sigma_1(k)} = 1$ for $k > n$, thus defining a sequence of signs $(\theta_l)_{l \in \N}$. Define $U \in \Iso(X)$ by $U(e_l) = \alpha_l e_{\tau(l)}$ for every $l \in \N$. We show that $\|UT_1 - T_2\| \leqslant \varepsilon$. First observe that for every $1\leqslant k \leqslant n$, 
    \begin{align*}
    \|UT_1(e_k) - T_2(e_k)\| & \leqslant \|U(T_1(e_k) - \theta_k^1 e_{\sigma_1(k)})\| + \|U(\theta_k^1 e_{\sigma_1(k)}) - T_2(e_k)\|\\
    &\leqslant \|T_1(e_k) - \theta_k^1 e_{\sigma_1(k)}\| + \|\theta_k^2 e_{\sigma_2(k)} - T_2(e_k)\| \leqslant \frac{\varepsilon}{n}.
    \end{align*}
    Knowing that the basis $(e_k)_{k=1}^n$ of $E_n$ is $1$-unconditional, we get, for every $x \in \Sphere_{E_n}$,
    \begin{equation*}
        \|UT_1(x) - T_2(x)\| \leqslant \sum_{k=1}^n |x(k)| \cdot \|UT_1(e_k) - T_2(e_k)\| \leqslant n \cdot \varepsilon/n = \varepsilon,
    \end{equation*}
    concluding the proof.
\end{proof}

We now consider the special case $M(t) = t^pe^{t-1}$, where $p \in (3, \infty)$ is fixed. The space $\ell_M$ is cofinally Fra\"iss\'e by Corollary \ref{cor:OrliczFraisse}. Since $e^{-1}t^p \leqslant M(t) \leqslant t^p$ for $t \in [0, 1]$, by \cite[Proposition 4.a.5]{LT}, the space $\ell_M$ is isomorphic to $\ell_p$. (Note that they are not isometric: indeed, as a consequence of \cite[Theorem 6.2]{CdRD}, $\ell_p$ is cannot be guarded Fra\"iss\'e, since it is bi-finitely representable with $L_p([0, 1])$, which is guarded Fra\"iss\'e.) In particular, $\ell_M$ cannot contain isomorphic copies of $\ell_2$, and hence is not oligomorphic.

\smallskip

To conclude the proof of Theorem \ref{thm:Counterexample}, what remains to be shown is that the age of $\ell_M$ is not closed. To start with, let us first recall the formal definition of the age. For $n \in \N$, denote by $\BM_n$ the set of all Banach spaces of dimension $n$, considered up to isometry. Endow it with the \textit{Banach--Mazur distance}:
$$d_{\BM}(E, F) \coloneq \inf\left\{\log\left(\|T\|\cdot\|T^{-1}\|\right) \,\big|\, T \colon E \to F \text{ is an isomorphism}\right\}.$$ It is a classical fact that this defines a metric on $\BM_n$, that the infimum in its definition is always attained, and that $\BM_n$ is compact for this metric. Let $\BM$ be the disjoint union of all the $\BM_n$'s (endowed with the disjoint union topology). The \textit{age} of a Banach space $X$, denoted by $\Age(X)$, is the set of all $E  \in \BM$ that can be isometrically embedded in $X$. We say that $X$ \textit{has closed age} if its age is a closed subset of $\BM$. Note that a Banach space $X$ is finitely representable in a Banach space $Y$ iff $\Age(X) \subseteq \overline{\Age(Y)}$.

\smallskip

We also recall a classical fact from the theory of Orlicz sequence spaces.

\begin{theorem}[\cite{LT}, Theorem~4.a.9]
 Suppose that $M$ satifies the $\Delta_2$-condition at zero. Then $\ell_p$ is isomorphic to a subspace of $\ell_M$ if and only if $p\in[\alpha_M,\beta_M]$ where 
 \[\alpha_M=\sup\, \biggl\{q;\, \sup_{0<s,t\leqslant 1}\frac{M(st)}{M(s)t^q}<\infty\biggr\},\]
\[\beta_M=\inf\, \biggl\{q;\, \inf_{0<s,t\leqslant 1}\frac{M(st)}{M(s)t^q}>0\biggr\}, \]
are the Boyd indices. 
\end{theorem}
\noindent Note that since $M$ satifies the $\Delta_2$-condition at zero, it follows that $\beta_M<\infty$. Moreover, it can be easily proved that $1\leqslant\alpha_M\leqslant \beta_M$.

 \begin{Remark}\label{fin.rep}
After the proof of \cite[Theorem~4.a.9]{LT},  the authors observe that if $p\in[\alpha_M,\beta_M]$ then $\ell_M$ contains almost-isometric copies of $\ell_p$, and hence $\ell_p$ is finitely representable in $\ell_M$. Alternatively, this finite representability follows from Krivine's theorem \cite{K}, which implies that $\ell_p$ is finitely representable in any isomorphic copy of itself.
\end{Remark}

The fact that $\Age(\ell_M)$ is not closed when $M(t) = t^pe^{t-1}$, $3 < p < \infty$, is an immediate consequence of the following proposition.

\begin{proposition}\label{age}
    Let $M$ be an Orlicz function that satifying the assumptions of Theorem \ref{thm:MainDisjointnessPreservation}. Suppose that there exists $3 < p < \infty$ and $C>0$ such that 
    \[C\leqslant \frac{M(st)}{M(s)t^p}<1, \;\; \text{for every}\;\; s\in (0,1]\, \text{and}\; t\in (0,1). \]
    Then, $\operatorname{Age}(\ell_M)$ is not closed.
\end{proposition}
\begin{proof}
 Under these hypotheses we have $\alpha_M=\beta_M=p$. 
Hence, by Remark~\ref{fin.rep}, $\ell_M$ contains almost-isometric copies of $\ell_p^2$.
To prove that $\operatorname{Age}(\ell_M)$ is not closed, it suffices to show that $\ell_p^2$ does not embed isometrically into $\ell_M$.
Suppose, for a contradiction, that there exists a linear isometric embedding $T:\ell_p^2\to\ell_M$, where $\ell_p^2=\operatorname{span}\{e_1,e_2\}$. 
Write $x_i:=T(e_i)$, for $i=1,2$. Note that $x_i\in S_{\ell_M}$, and by Corollary \ref{cor:DisjointnessPreservation}, those vectors are disjoint.
Let $a:= 2^{-1/p}$. Since $ae_1+ae_2\in S_{\ell_p}$, it follows that 
  \begin{align}
1
&= \sum_{k=1}^\infty M\bigl(a|x_1(k)+x_2(k)|\bigr) \notag\\
&= \sum_{k=1}^\infty M\bigl(a|x_1(k)|\bigr)
      + \sum_{k=1}^\infty M\bigl(a|x_2(k)|\bigr)\notag\\
&< a^p \sum_{k=1}^\infty M\bigl(|x_1(k)|\bigr)
   + a^p \sum_{k=1}^\infty M\bigl(|x_2(k)|\bigr) \notag\\
&= 2a^p =1,\notag
\end{align}
which is a contradiction.
  \end{proof}

\bibliographystyle{plain}
\bibliography{main}
\end{document}